\documentclass[a4paper,fleqn,10pt]{article}
\usepackage{a4wide}
\usepackage[pdftex]{color}
\definecolor{urlcolor}{rgb}{0,0.5,0}
\definecolor{linkcolor}{rgb}{0.5,0,0}
\definecolor{citecolor}{rgb}{0,0,0.5}
\usepackage[pdfpagemode=UseOutlines,urlcolor=urlcolor,linkcolor=linkcolor,citecolor=citecolor,colorlinks=true]{hyperref}

\usepackage{tikz}
\usepackage{amsfonts}
\usepackage{amsmath}
\usepackage{amssymb}
\usepackage{enumerate}
\usepackage[british]{babel}
\usepackage{natbib}
\usepackage[amsmath,thref,thmmarks,hyperref]{ntheorem}
\usepackage{graphicx}

\usetikzlibrary{through}
\usetikzlibrary{decorations.pathreplacing}

\usepackage{color}
\definecolor{purple}{rgb}{0.5,0,0.5}
\definecolor{darkgreen}{rgb}{0,0.6,0}

\newcommand{\N}{\mathbb{N}}

\newcommand{\R}{\mathbb{R}}
\newcommand{\rp}{\mathbb{R}\mathrm{P}}
\newcommand{\C}{\mathbb{C}}
\newcommand{\cp}{\mathbb{C}\mathrm{P}}
\newcommand{\E}{\mathbf{E}}
\renewcommand{\P}{\mathbf{P}}
\newcommand{\herm}{\mathrm{Herm}}
\newcommand{\gr}{\mathrm{Gr}}

\DeclareMathOperator{\argmin}{\text{argmin}\,}

\newtheorem{thm}{Theorem}
\newtheorem{prop}[thm]{Proposition}
\newtheorem{cor}[thm]{Corollary}

\newtheorem{rem}[thm]{Remark.}

\theorembodyfont{\upshape}
\theoremheaderfont{\scshape}
\theoremstyle{nonumberplain}

\theoremsymbol{\raisebox{1pt}{\scriptsize\ensuremath{\Box}}}
\newtheorem{proof}{Proof.}
\newtheorem{proof2}{Proof of \thref{grprojballuniq}.}

\newcommand{\rinline}[1]{???}

\begin{document}
\title{Non-asymptotic Confidence Sets for Extrinsic Means\linebreak on Spheres and Projective Spaces}
\author{Thomas Hotz, Florian Kelma
\\[12pt] Institut f\"ur Mathematik\\ Technische Universit\"at Ilmenau}
\date{}
\maketitle

\begin{abstract}
Confidence sets from i.i.d.\ data are constructed for the extrinsic mean of a probabilty measure $\P$ on spheres, real projective spaces $\rp^k,$ and complex projective spaces $\cp^k,$ as well as Grassmann manifolds, with the latter three embedded by the Veronese-Whitney embedding. When the data are sufficiently concentrated, these are projections of a ball around the corresponding Euclidean sample mean. Furthermore, these confidence sets are rate-optimal. The usefulness of this approach is illustrated for projective shape data.
\end{abstract}

\section{Introduction and definitions} 

Data assuming values on unit spheres, particularly on the circle, as well as on real projective spaces arise frequently in applications, examples being measurements of wind directions or axial data describing optical axes of crystals \citep{MJ}. Moreover, we consider data assuming values on real and complex Grassmann manifolds which play important roles in several shape spaces which arise in image-processing applications, cf.\ \citep{KBCL,PM,MP}.

Motivated by these applications, we will concern ourselves with the arguably simplest statistic, the mean. However, since there is neither a natural addition of points on these manifolds nor the division by a natural number, the meaning of ``mean'' is a priori unclear.


We assume the data can be modelled as independent random variables $Z_1,\dots,Z_n$ which are identically distributed as the random variable $Z$ taking values in a closed subset $M$ of the unit sphere $S^k = \{ x \in \R^{k+1} \,:\, \| x \| = 1 \}.$ Using the Veronese-Whitney embedding, real and complex Grassmann manifolds $\gr(m, \R^{d+1})$ resp.\ $\gr(m, \C^{d+1}),$ $m<d+1,$ can be seen as closed subspaces of unit spheres allowing to treat these cases, too. Recall that $\gr(1, \R^{d+1}) = \rp^d$, and analogously $\gr(1, \C^{d+1}) = \cp^d$ are the real and complex projective spaces of dimension $d.$


Since any unit sphere is a subset of a Euclidean vector space, the Euclidean sample mean $\bar{Z}_n = \tfrac{1}{n}\sum_{i=1}^n {Z_i}$ is well-defined, but $\bar{Z}_n$ cannot be taken as a mean of the sphere since it may have norm less than 1. Though, the Euclidean sample mean is the minimiser of the sum of squared distances, and thus this can be put in the more general framework of Fr\'{e}chet means, see \citep{MF}: define the set of extrinsic sample means to be \begin{equation}
\hat{\mu}_n = \underset{\nu\in M}{\argmin} \sum_{i=1}^n\| Z_i - \nu \|^2,
\end{equation}
and analogously define the set of extrinsic population means of the random variable $Z$ to be
\begin{equation*} \mu = \underset{\nu\in M}{\argmin} \E \| Z - \nu \|^2 \end{equation*} where $M\subseteq S^k$ is the subset of the possible values of $Z.$ As usual, the extrinsic sample means are the extrinsic population means when considering the empirical distribution of $Z_1,\dots,Z_n.$

The extrinsic population mean is closely related to the Euclidean population mean $\E Z$: \begin{equation*} \mu = \underset{\nu\in M}{\argmin} \|\E Z - \nu \|^2 \end{equation*} since $\E\|Z-\nu\|^2 = \E\|Z-\E Z\|^2 + \|\E Z-\nu\|^2.$ Hence, $\mu$ is the set of points on the manifold closest to $\E Z,$ and $\mu$ is unique if and only if the orthogonal projection of $\E Z$ onto $M$ is unique; for $M=S^k,$ this is the case if and only if $\E Z\neq0,$ since the orthogonal projection of $\E Z$ is then given by $\mu=\E Z/\|\E Z\|$; otherwise, i.e.\ if $\E Z=0,$ the set of extrinsic population means is all of $S^k.$ Analogous results hold for the extrinsic sample mean.

In this article, we aim to construct non-asymptotic confidence sets for the extrinsic population mean without making any assumptions about the distribution of the data besides the assumption that the data are independent and identically distributed. It has been shown in \citep{Hot} and \citep{HKW} that this is possible for data on the circle. We will generalize and improve the formalism used in \citep{Hot} to obtain confidence sets in this more general case.

First, we construct confidence sets for the Euclidean population mean noting that the projection of these sets are confidence sets for the extrinsic population mean. The constructed confidence sets for the Euclidean population mean are no open balls, but we will see (\thref{proconf}) that a $(1-\alpha)$-confidence set for the extrinsic population mean is given by the projection of the open ball around the Euclidean sample mean $\bar{Z}_n$ with radius \begin{equation*} \varepsilon = \sqrt{\tfrac{1}{\alpha n} \left(1-\|\bar{Z}_n\|^2 + \tfrac{1}{\alpha n} \right) }. \end{equation*}So, the remaining problem for arbitrary closed $M\subseteq S^k$ is to understand the projection of open balls.

In the simple case where $M=S^k,$ we immediately obtain \begin{equation*} \left\{ x\in S^k \, : \, \sin \sphericalangle (x,\hat{\mu}_n) < \sqrt{ \tfrac{1-\left\|\bar{Z}_n\right\|^2+\tfrac{1}{\alpha n}}{\alpha n \left\|\bar{Z}_n\right\|^2}} \right\} \end{equation*} as a $(1-\alpha)$-confidence set, if $\| \bar{Z}_n \|^2 \geq \tfrac{1}{\alpha n}$ (\thref{sphereconf}). We further note that this condition is violated with exponentially decreasing probability for $\E Z \neq 0,$ and fulfilled with probability at most $\alpha$ for $\E Z = 0.$

Additionally, we cover complex Grassmann manifolds $\gr(m,\C^{k+1})$ in  \hyperref[grassmann]{Section\thref{grassmann}}. These spaces can be considered as closed subspaces of a $((k+1)^2-1)$-dimensional sphere by the \emph{Veronese-Whitney embedding} into the Euclidean space of Hermitian matrices with the Euclidean norm $\|\cdot\|$ being called Frobenius norm $\|\cdot\|_F$ there. In this case, the projection $\pi$ maps a Hermitian matrix to the span of eigenvectors corresponding to its $m$ largest eigenvalues. Using basic results from linear algebra, we obtain an open $(1-\alpha)$-confidence ball around the extrinsic sample mean $\hat{\mu}_n$ with radius \begin{equation*} \delta_n = \frac{\sqrt{2} \sqrt{1-\|\bar{Z}_n\|_F^2+\tfrac{1}{\alpha n}}} {\sqrt{\alpha n}(\hat{\lambda}_m-\hat{\lambda}_{m+1})-\sqrt{1-\|\bar{Z}_n\|_F^2+\tfrac{1}{\alpha n}}}\end{equation*} if the $m$-th largest and $m+1$-st largest eigenvalues $\hat{\lambda}_m$ and $\hat{\lambda}_{m+1}$ of the Euclidean sample mean $\bar{Z}_n$ are sufficiently separated (\thref{grconfreg}). Again, this condition is violated with exponentially decreasing probability for uniquely projected $\E Z ,$ and fulfilled with probability at most $\alpha$ for non-uniquely projected $\E Z .$

This can then be applied to shape data, examples of which are shown in \hyperref[Applications]{Section 5}. We conclude with a discussion of the results.

\section{Confidence sets for Euclidean and extrinsic means}\label{GeneralSetting}

As above, let $Z_1,\ldots,Z_n, \,n\in\N$ be i.i.d.\ random elements on $M$ where $M$ is a closed subset $M$ of the \emph{unit sphere} $S^k = \{ x \in \R^{k+1} \,:\, \vert x \vert = 1 \}.$  Further, let $\pi: \R^{k+1}\rightarrow M$ be the projection of $\R^{k+1}$ to $M$ in the sense of best approximation. 
With this setup, we have the following multivariate version of Chebyshev's inequality:

\begin{thm}\label{thmmarkov}
For $\alpha\in(0,1),$ \begin{equation*} \mathbf{P}\left(\left\|\bar{Z}_n-\mathbf{E}Z\right\|^2\geq\frac{1-\left\|\mathbf{E}Z\right\|^2}{\alpha n}\right)\leq\alpha.\end{equation*}
\end{thm}

\begin{proof} Observe that 
$\mathbf{E}\left\|Z-\mathbf{E}Z\right\|^2=1-\left\|\mathbf{E}Z\right\|^2$, 
such that by the assumption of observing i.i.d.\ data, $\mathbf{E}\left\|\bar{Z}_n-\mathbf{E}Z\right\|^2=\frac{1}{n}(1-\left\|\mathbf{E}Z\right\|^2)$. 
Now apply Markov's inequality.\end{proof}

This simple inequality gives a level $\alpha$ \emph{test of the hypothesis} $\E Z = q$ for any fixed $q$ with $\Vert q \Vert \leq 1$ (note that necessarily $\Vert \E Z \Vert \leq 1$ by convexity): simply reject $q$ if $\| \bar{Z}_n - q \|^2 \geq \frac{1-\|q\|^2}{\alpha n}.$ The acceptance set of that test, 
\begin{equation*} C(p,a) := \left\{ q \,:\, \Vert q \Vert \leq 1 \,\wedge\, \|p-q\|^2<\frac{1-\|q\|^2}{a} \right\} \end{equation*} 
with $p=\bar{Z}_n$ and $a=\alpha n,$ is therefore a $(1-\alpha)$-confidence set for the Euclidean population mean. Henceforth, the projection of this set under $\pi$ is a $(1-\alpha)$-confidence set for the extrinsic population mean.

\begin{prop}\label{proball}
For $r>0,$ define the open ball \begin{equation*}B_r(p):=\left\{q\in E\,:\,\|p-q\|<r\right\}\end{equation*} around $p$ with radius $r.$  Then \begin{equation*}\pi\left(C(p,a)\right)=\pi\biggl(B_{\sqrt{\tfrac{1}{a} \left(1-\|p\|^2 + \tfrac{1}{a} \right) } }(p)\biggr).\end{equation*} In particular, if $\|p\|^2 \geq \tfrac{1}{a},$ then $C(p,a)$ is contained in the positive cone over \begin{equation*} B_{\sqrt{\tfrac{1}{a} \left(1-\|p\|^2 + \tfrac{1}{a} \right) } }(p).\end{equation*}\\
\end{prop}

\begin{proof}

\begin{figure}[t]
\centering
\begin{tikzpicture}

\clip (-0.5,-0.5) rectangle (10,4.5);
\draw (0,0) circle (8cm);
\draw (10:4.5cm) -- (0,0); 
\draw[dashed] (10:8cm) -- (0,0);
\draw (0,0) -- (30:6.5cm); 
\draw[dashed] (30:6.5cm) -- (30:8cm);
\draw[blue] (10:4.5cm) -- (30:6.5cm) node [pos=0.45,left]{\footnotesize $\sqrt{\frac{1-\left\|q\right\|^2}{a}}$}; 
\node [draw=blue,dashed] at (10:4.5cm) [circle through={(30:6.5cm)}]{};
\draw (10:0.8cm) arc (10:30:0.8cm);
\draw [white] (20:01.2cm) node[black]{\footnotesize{$\varphi(q)$}};
\draw (10:5.864cm) -- (12.35:5.87cm) -- (12.27:6.1091); 
\draw[green!50!black] (30:6.5cm) -- (10:6.10800203511cm) node[pos=0.5,right]{\footnotesize{$\varepsilon$}};
\fill (0,0) circle (1.5pt) node [left] {\small $0$}; 
\fill (10:4.5cm) circle (1.5pt) node [anchor=south east]{\footnotesize $q$}; 
\fill (10:8cm) circle (1.5pt) node [right] {$\pi(q)$}; 
\fill (30:8cm) circle (1.5pt) node [right] {$\pi(p)$}; 
\fill (30:6.5cm) circle (1.5pt) node [above]{\footnotesize $p$};

\end{tikzpicture}
\caption{Sketch for the proof of \hyperref[proball]{Proposition \ref{proball}}.}
\label{Fig.1}
\end{figure}
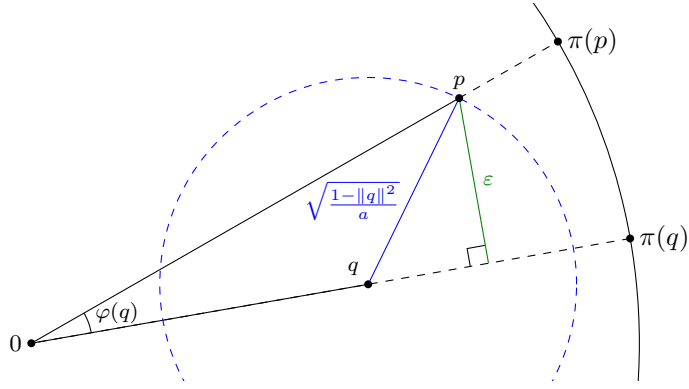
We will bound the angle $\varphi$ between $p$ and arbitrary $q\in C(p,a).$ The angle is given by \begin{equation*}\varphi = \arccos \frac{\left\langle p,q\right\rangle}{\|p\| \|q\|}, \end{equation*} hence we will compute the minimum of $ \cos\varphi = \tfrac{\left\langle p,q\right\rangle}{\|p\| \|q\|}$ with respect to $q$ under the constraint 
\begin{align*} & \|p-q\|^2  = \frac{1-\|q\|^2}{a} \\
\Leftrightarrow\quad & \|p\|^2-2\left\langle p,q\right\rangle + \|q\|^2  = \frac{1-\|q\|^2}{a} \\
\Leftrightarrow\quad & \left\langle p,q\right\rangle  = \frac{1}{2} \left( \|p\|^2 + \|q\|^2 - \frac{1-\|q\|^2}{a} \right),
\end{align*} see \hyperref[Fig.1]{Figure 1}.
Hence, \begin{equation*} \frac{\|p\|^2 + \|q\|^2 - \frac{1-\|q\|^2}{a}}{2\|p\| \|q\|} = \frac{\|p\|^2 + \frac{a+1}{a}\|q\|^2 - \frac{1}{a} }{2\|p\| \|q\|}
\end{equation*} has to be minimized with respect to $\|q\|.$ Starting from 
\begin{equation*}
0 = \frac{\partial}{\partial \|q\|} \left( \frac{\|p\|^2 + \frac{a+1}{a}\|q\|^2 - \frac{1}{a} }{2\|p\| \|q\|} \right) \Bigg|_{q_{max}} = \frac{1}{2\|p\|} \frac{2\frac{a+1}{a} \|q_{max}\|^2 - \|p\|^2-\frac{a+1}{a} \|q_{max}\|^2 +\frac{1}{a}}{\|q_{max}\|^2},
\end{equation*}
one obtains \begin{equation*} \|q_{max}\| = \sqrt{\frac{a}{a+1}\left(\|p\|^2-\frac{1}{a}\right)}, \end{equation*}
and thus \begin{align*}
\varepsilon^2_{max} & = \|p\|^2\sin^2\varphi_{max} \\
& = \|p\|^2 (1 - \cos^2 \varphi_{max}) \\
& = \frac{1}{4\|q_{max}\|^2} \left( 4\|p\|^2\|q_{max}\|^2  - \left( \|p\|^2 + \frac{a+1}{a}\|q_{max}\|^2 - \frac{1}{a} \right)^2 \right) \\
& = \frac{a+1}{4a (\|p\|^2-\tfrac{1}{a})} \left( 4\frac{a}{a+1}\|p\|^2\left(\|p\|^2-\frac{1}{a}\right) - 4 \left(\|p\|^2-\frac{1}{a}\right)^2 \right) \\
& = \frac{a+1}{a (\|p\|^2-\tfrac{1}{a})} \left( \|p\|^2-\frac{1}{a} \right) \left( \frac{a}{a+1}\|p\|^2 - \|p\|^2 + \frac{1}{a} \right) \\
& = -\frac{1}{a}\|p\|^2 + \frac{a+1}{a^2} \\
& = \frac{1}{a}\left( 1-\|p\|^2+\frac{1}{a} \right)
\end{align*}
which finishes the proof.
\end{proof}

\begin{cor}\label{proconf}
A $(1-\alpha)$-confidence set for the extrinsic population mean is given by the projection of the ball around $\bar{Z}_n$ with radius \begin{equation*} \varepsilon = \sqrt{\tfrac{1}{\alpha n} \left( 1-\|\bar{Z}_n\|^2  + \tfrac{1}{\alpha n} \right) }.\end{equation*}
More precisely, the probability that $\pi(B_\varepsilon(\bar Z_n))$ covers the extrinsic mean set is at least $1 - \alpha$.

In addition, $0 \in B_\varepsilon(\bar Z_n)$ if and only if $\Vert \bar Z_n \Vert^2 < \frac{1}{\alpha n}$. In that case, $\pi(B_\varepsilon(\bar Z_n)) = M$.
\end{cor}

\begin{rem}
As an alternative to the use of Chebychev's inequality, a multivariate version of Hoeffding's inequality could be used to get confidence sets. After some calculations following the lines of \citep[Ex.~6.3]{BLM}, one can obtain 
\begin{equation*} \P \left(\|\bar{Z}_n - \E Z\| > \tfrac{1}{\sqrt{n}}\left(\sqrt{-2 (1+\|\E Z\|)^2 \log\alpha} + \sqrt{1-\| \E Z \|^2} \right) \right) \leq \alpha. \end{equation*} 
For $\alpha = .05$ and $\E Z = 0,$ this radius is roughly 23\% smaller than the corresponding radius obtained by Chebyshev's inequality. However, the use of Chebyshev's inequality gives smaller confidence sets for $\|\E Z\|$ close to~1 which is important since directional data are often highly concentrated.
\end{rem}

\section{Confidence sets for directional data} \label{sphere}

In this section, let $M = S^k := \left\{x\in\mathbb{R}^{k+1}\,:\,\|x\|=1\right\}$ be the entire $k$-dimensional sphere. Clearly, $\pi$ is now given by \begin{equation*}\pi(x) = \underset{\nu\in S^k}{\text {argmin}} \| \nu-x\| = \frac{x}{\|x\|}\end{equation*} for $x\neq0,$ while $\pi(0)=S^k.$ The following statement is well-known; see e.g.\ \citep{BP03} or \citep{MJ}.

\begin{thm}\label{spheremean}
The extrinsic population mean $\mu$ of $\mathbf{P}^Z$ is unique if and only if $\mathbf{E}Z\neq 0.$ In this case, $\mu=\frac{\mathbf{E}Z}{\|\mathbf{E}Z\|},$ otherwise $\mu=S^k.$ Similarly, the extrinsic sample mean $\hat{\mu}_n$ is unique if and only if $\bar{Z}_n=\frac{1}{n}\sum_{i=1}^n{Z_i} \neq 0,$ whence $\hat{\mu}_n = \frac{\bar{Z}_n}{\|\bar{Z}_n\|},$ otherwise $\hat{\mu}_n = S^k.$
\end{thm}

Using \thref{proconf}, we immediately obtain confidence sets for the extrinsic population mean.

\begin{prop}\label{sphereconf}
Let $\alpha\in(0,1).$ Then $\pi(C(\bar{Z}_n,\alpha n))$ is a $(1-\alpha)$-confidence set for the extrinsic population mean (set) of directional data.
\begin{enumerate}[(i)]
  \item If $\|\bar{Z}_n\|^2 \geq \tfrac{1}{\alpha n},$ then $\pi(C(\bar{Z}_n,\alpha n))$ is given by \begin{equation*} \left\{ x\in S^k \, : \, \sin \sphericalangle (x,\hat{\mu}_n) < \sqrt{ \tfrac{1-\left\|\bar{Z}_n\right\|^2+\tfrac{1}{\alpha n}}{\alpha n \left\|\bar{Z}_n\right\|^2}} \right\}, \end{equation*} which, in the metric on $S^k$ given by arc length, is the open ball around $\hat{\mu}_n$ with radius \begin{equation*}\arcsin \sqrt{ \tfrac{1-\left\|\bar{Z}_n\right\|^2+\tfrac{1}{\alpha n}}{\alpha n \left\|\bar{Z}_n\right\|^2}} .\end{equation*} Otherwise, $\pi(C(\bar{Z}_n,\alpha n)) = S^k$, the closed ball around $\hat\mu_n$ with radius $\pi$.
	\item If $\E Z=0$, i.e.\ if the extrinsic population mean set equals $S^k,$ then $\|\bar{Z}_n\|^2  \geq \tfrac{1}{\alpha n}$ with probability at most $\alpha,$ so indeed $\pi(C(\bar{Z}_n,\alpha n)) = S^k$ with probability at least $(1-\alpha)$.
	\item If $\E Z \neq 0$, i.e.\ if the extrinsic population mean is unique, then \begin{equation*} \P \bigl( \|\bar{Z}_n\| \leq \tfrac{1}{\alpha n} \bigr) \leq e^{- \tfrac{n}{2} \bigl( \|\E Z\| - \tfrac{1}{\alpha n} \bigr)^2} 
	\,, \end{equation*}
	hence the probability of obtaining the trivial confidence set $\pi(C(\bar{Z}_n,\alpha n)) = S^k$ goes to zero at an exponential rate when $n$ tends to infinity. Furthermore, the radius of the confidence ball in (i) is in $O_\P \bigl(\tfrac{1}{\sqrt{n}}\bigr)$.
\end{enumerate}
\end{prop}

\begin{proof} 
For \emph{(i)}, see  \thref{proconf}, and note that the projection $\pi$ preserves the angle between two points. If $\|\bar{Z}_n\|^2 \geq \tfrac{1}{\alpha n},$ then the boundary for the angle between a point in the $(1-\alpha)$-confidence set for the Euclidean mean and the Euclidean sample mean is given by \begin{equation*}\sin \varphi_{max} = \frac{\sqrt{\tfrac{1}{\alpha n} \left( 1-\|\bar{Z}_n\|^2+\tfrac{1}{\alpha n} \right) } }{ \|\bar{Z}_n\| }.\end{equation*} Hence, this is the boundary for the angle after projecting to $S^k.$ If $\|\bar{Z}_n\|^2 < \tfrac{1}{\alpha n},$ then $0\in C(\bar{Z}_n,\alpha n)$ whence $\pi(C(\bar{Z}_n,\alpha n)) = S^k.$ 

\emph{(ii)} follows directly from \thref{thmmarkov}.

For \emph{(iii)}, we use Hoeffding's inequality, cf.\ \citep{Hoe}: \begin{align*} \P \bigl( \| \bar{Z}_n \| \leq \frac{1}{\alpha n} \bigr) & \leq \P \bigl( | \left\langle \bar{Z}_n , \tfrac{\E Z}{\|\E Z\|} \right\rangle | \leq \frac{1}{\alpha n} \bigr) \\
& \leq \P \bigl( \left| \left\langle \bar{Z}_n , \tfrac{\E Z}{\|\E Z\|} \right\rangle -\| \E Z\| \right| \geq \| \E Z \| - \frac{1}{\alpha n} \bigr) \\
& \leq e^{- \tfrac{n}{2} \left( \| \E Z \| - \tfrac{1}{\alpha n} \right)^2}
\end{align*}
Then, the second statement of \emph{(iii)} follows directly from \emph{(i)}.
\end{proof} 


For data on the circle $S^1$ however, we suggest to construct confidence sets using Hoeffding's inequality as in \citep{HKW}.

\begin{rem}
Using the central limit theorem for the extrinsic mean, cf.\ \citep{JS}, one can obtain asymptotic confidence sets \begin{equation*}
 \left\{ x\in S^k: \sphericalangle (x,\hat{\mu}_n) < q_{1-\tfrac{\alpha}{2}} \sqrt{\frac{\tfrac{1}{n} \sum_{i=1}^n{\sin^2 \sphericalangle (Z_i,\bar{Z}_n)}}{n\|\bar{Z}_n\|^2}} \right\}
\end{equation*} where $q_{1-\tfrac{\alpha}{2}}$ is the $(1-\tfrac{\alpha}{2})$-quantil of the standard normal distribution, but these do not guarantee coverage for finite $n$ which is problematic even for considerably large sample sizes, cf.\ \citep{HKW}.
\end{rem}


\section{Confidence sets for data on Grassmann manifolds}\label{grassmann}

In this section, let $\gr(m,\C^{d+1}),$ $m<d+1,$ be the Grassmann manifold of $m$-dimensional complex subspaces of $\C^{d+1}.$ Analogous results hold for real Grassmann manifolds. The projective spaces $\cp^d=\gr(1,\C^{d+1})$ and $\rp^d=\gr(1,\R^{d+1})$ are special cases of Grassmann manifolds\footnote{Note that $\rp^1\cong S^1$ and $\cp^1\cong S^2.$ We suggest using the results of \hyperref[sphere]{Section 3} to obtain confidence sets from data in these spaces since a lower dimensional Euclidean space is used there.}.

We embed the Grassmann manifold $\gr(m,\C^{d+1})$ into the $(d+1)^2$-dimensional Euclidean space $\big(\herm(d+1), tr(\cdot ,\cdot)\big)$ of Hermitian matrices by choosing an orthonormal basis $\{u_1,\ldots,u_m\}$ for a given subspace $U\in\gr(m,\C^{d+1})$ and mapping it to \begin{equation*}
\iota(U) = \frac{1}{\sqrt{m}}\sum_{i=1}^m {u_i u_i^*}
\end{equation*} which is independent of the choice of the basis, making the mapping well-defined; $\iota:\gr(m,\C^{d+1})\rightarrow\herm(d+1)$ is called Veronese-Whitney embedding. As a mapping, $\iota(U)$ is actually the projection from $\C^{k+1}$ onto the subspace $U,$ so $\iota$ is injective, allowing us to identify $\gr(m,\C^{d+1})$ with $M=\iota(\gr(m,\C^{d+1})).$ Note that \thref{proconf} holds for data in $\gr(m,\C^{k+1})$ since $\|\iota(U)\|_F= \sqrt{tr(\iota(U)\iota(U)^*)}=1$ for all $U\in \gr(m,\C^{d+1})$ with $\|\cdot\|_F$ being the Frobenius norm, i.e.\ the standard Euclidean norm on $\C^{(d+1)\times(d+1)}$ viewed as $\C^{(d+1)^2}.$ Therefore, $M\subseteq S^k$ for $k=(d+1)^2-1,$ allowing us to apply the results of \hyperref[GeneralSetting]{Section2}.

The Frobenius norm then equips $\gr(m,\C^{k+1})$ with an extrinsic metric $d$ defined by \begin{align*}
d^2 (U,V) & = \| \iota(U)-\iota(V) \|_F^2 \\
& =\frac{1}{m} \biggl( \sum_{i=1}^m tr(u_i u_i^*) + \sum_{i=1}^m tr(v_i v_i^*) - 2\sum_{i,j=1}^m tr(u_i u_i^* v_j v_j^*) \biggr) \\
& =2(1- \frac{1}{m} \sum_{i,j=1}^m |v_j^*u_i|^2) \\
& = 2(1 - \| \iota(U)\iota(V) \|_F^2). 
\end{align*}

As always, $\pi:\herm(k+1)\rightarrow M$ shall be the projection in the sense of best approximation. Here, $\pi$ maps a matrix to the span of $m$ linear independent eigenvectors to its largest eigenvalues. The image of matrix under $\pi$ is unique if and only if the direct sum of the eigenspaces corresponding to the $m$ largest eigenvalues (counted with multiplicities) is $m$-dimensional.

\begin{prop}\label{grprojballuniq}
Let $A$ be a Hermitian matrix with eigenvalues $\lambda_1 \geq\ldots\geq \lambda_m > \lambda_{m+1}\geq\ldots\geq\lambda_{k+1},$ so that $\pi(A)$ is unique. If $\lambda_m - \lambda_{m+1} \geq \sqrt{2}\varepsilon>0,$ then $\pi$ is unique on $B_\varepsilon(A)\subset \herm(k+1).$
\end{prop}

For the proof, we need the theorem of Wielandt-Hoffman, cf.\ \citep{WH}:

\begin{thm}\label{wielandt} (Wielandt-Hoffman)
If $A$ and $A+E$ are normal matrices with eigenvalues $\lambda_1\geq\ldots\geq\lambda_{k+1}$ respectively $\sigma_1\geq\ldots\geq\sigma_{k+1}$, then $$\sum_{j=1}^{k+1}{\left(\lambda_i - \sigma_i\right)^2}\leq\|E\|_F^2.$$
\end{thm}

\begin{proof2}
Let $\sigma_1\geq\ldots\geq\sigma_{k+1}$ be the eigenvalues of $A+E\in B_\varepsilon(A).$
It has to be shown that $ \sigma_m-\sigma_{m+1} > 0 $ for all $A+E\in B_\varepsilon (A),$ i.e.\ for all $E\in\herm(k+1)$ with $ \|E\|_F < \varepsilon. $

Suppose that $\sigma_m = \sigma_{m+1}.$ Then 
\begin{equation*}\left(\lambda_m-\sigma_m\right)^2+\left(\lambda_{m+1}-\sigma_{m+1}\right)^2\end{equation*}
 is minimal for 
 $\sigma_m=\sigma_{m+1}=\frac{\lambda_m+\lambda_{m+1}}{2}.$ 
 Hence, \begin{equation*} \|E\|_F^2 = \sum_{j=1}^{k+1}{\left(\lambda_i - \sigma_i\right)^2} \geq \left(\lambda_m-\sigma_m\right)^2+\left(\lambda_{m+1}-\sigma_{m+1}\right)^2 \geq \varepsilon^2,\end{equation*} but that contradicts the  \hyperref[wielandt]{Wielandt-Hoffman Theorem} since $\|E\|_F<\varepsilon.$
\end{proof2}

To bound the projection of the ball in \thref{proconf}, we use the Davis-Kahan $\sin \theta$ Theorem (cf.\ \citep{DK}), and conclude the following statements:

\begin{thm}\label{daviskahan}
Let $A$ and $A+E$ be Hermitian matrices with eigenvalues $\lambda_1 \geq\ldots\geq \lambda_m > \lambda_{m+1}\geq\ldots\geq\lambda_{k+1}$ resp.\ $\sigma_1 \geq\ldots\geq \sigma_m > \sigma_{m+1}\geq\ldots\geq\sigma_{k+1}$ with $\sigma_m > \lambda_{m+1}.$ Then, \begin{equation*}
(\sigma_m-\lambda_{m+1})\cdot d( \pi(A),\pi(A+E) ) \leq \sqrt{2} \| E\|_F. \end{equation*}
\end{thm}

\begin{prop}\label{grprojball}
Let $A$ be a Hermitian matrix with eigenvalues $\lambda_1 \geq\ldots\geq \lambda_m > \lambda_{m+1}\geq\ldots\geq\lambda_{k+1}.$ If $\lambda_m-\lambda_{m+1} \geq\sqrt{2}\varepsilon>0,$ then \begin{equation*}\pi\left(B_\varepsilon(A)\right)\subseteq B_\delta(\pi(A))\end{equation*}
in $\gr(m,\C^{k+1})$ where $\delta=\frac{\sqrt{2}\varepsilon}{\lambda_m-\lambda_{m+1}-\varepsilon}.$ 
\end{prop}

\begin{proof}
First, note that $B_\varepsilon(A)$ is uniquely projected by $\pi$ due to \thref{grprojballuniq} since $\lambda_m-\lambda_{m+1} \geq\sqrt{2}\varepsilon.$
Let $E$ be an arbitrary Hermitian matrix with $\|E\|_F<\varepsilon,$ i.e. $A+E\in B_\varepsilon(A),$ and denote the eigenvalues of $A+E$ by $\sigma_1\geq\ldots\geq\sigma_{k+1}.$  By construction and the Wielandt-Hoffman Theorem, \begin{equation*}\sigma_m-\lambda_j=\underbrace{\sigma_m-\lambda_m}_{>-\varepsilon}+\lambda_m-\lambda_j>\lambda_m-\lambda_{m+1}-\varepsilon\end{equation*} for all $j>m.$
Using \thref{daviskahan},
\begin{equation*}
d \left( \pi(A),\pi(A+E) \right)  \leq \frac{\sqrt{2} \|E\|_F}{\sigma_m - \lambda_{m+1}} < \frac{\sqrt{2}\varepsilon}{\lambda_m-\lambda_{m+1}-\varepsilon},
\end{equation*}
which is what had to be shown.
\end{proof}

Using these results and \thref{proconf}, we obtain confidence sets for the extrinsic population mean on $\gr(m,\C^{d+1}).$

\begin{prop}\label{grconfreg}
Let $\alpha\in(0,1),$ and suppose $\bar{Z}_n$ to have eigenvalues $\hat{\lambda}_1\geq\hat{\lambda}_2\geq\ldots\geq\hat{\lambda}_{k+1}.$
\begin{enumerate}[(i)]

  \item If \begin{equation*}\hat{\lambda}_m-\hat{\lambda}_{m+1} \geq\sqrt{\frac{2\left(1-\|\bar{Z}_n\|_F^2+\tfrac{1}{\alpha n}\right)}{\alpha n}},\end{equation*} then $\pi(C(\bar{Z}_n,\alpha n))$ is given by $D = B_{\delta_n}(\hat{\mu}_n)\subset\gr(m,\C^{k+1})$ where \begin{equation*}\delta_n=\frac{\sqrt{2}\sqrt{1-\|\bar{Z}_n\|_F^2+\tfrac{1}{\alpha n}}}{\sqrt{\alpha n}(\hat{\lambda}_m-\hat{\lambda}_{m+1})-\sqrt{1-\|\bar{Z}_n\|_F^2+\tfrac{1}{\alpha n}}}.\end{equation*} Choosing $D = M,$ the closed ball around $\hat{\mu}_n$ with radius $\sqrt{2},$ otherwise ensures that the extrinsic mean $\mu$ is covered by $D$ with probability at least $(1-\alpha).$
	
	\item If $\E Z$ is not uniquely projected, then \begin{equation*}\hat{\lambda}_m-\hat{\lambda}_{m+1} < \sqrt{ \frac{2 \left( 1-\|\bar{Z}_n\|_F^2+\tfrac{1}{\alpha n} \right) }{\alpha n}}\end{equation*} with probability at least $(1-\alpha).$ 
	
	\item If  there is a unique extrinsic population mean, then \begin{equation*} \P \left( \hat{\lambda}_m-\hat{\lambda}_{m+1} < \sqrt{ \frac{2 \left( 1-\|\bar{Z}_n\|_F^2+\tfrac{1}{\alpha n} \right) }{\alpha n}} \right) \leq e^{-\tfrac{1}{2} \left( \sqrt{n}\tfrac{\sigma_m-\sigma_{m+1}}{\sqrt{2}} - \sqrt{\tfrac{1+\tfrac{1}{\alpha n}}{\alpha}} - \sqrt{1-\|\E Z\|} \right)^2 } \overset{n\rightarrow\infty}{\longrightarrow} 0 \end{equation*}
where $\sigma_m,\,\sigma_{m+1}$ are the $m$-th and $(m+1)$-th eigenvalues of $\E Z.$	In particular, the probability of obtaining the trivial confidence set goes to zero at an exponential rate when $n$ tends to infinity. Furthermore, the radius of the confidence ball in (i) is in $O_\P \bigl(\tfrac{1}{\sqrt{n}} \bigr).$
\end{enumerate}
\end{prop}

\begin{proof} \emph{(i)} follows directly from \thref{proconf} and \thref{grprojball}.

For \emph{(ii)}, recall the \hyperref[wielandt]{Wielandt-Hoffman Theorem} and note that \begin{align*}
\quad & \lambda_m-\lambda_{m+1} - (\sigma_m-\sigma_{m+1}) \leq \sqrt{2} \| \E Z -\bar{Z}_n \|_F \\
\Leftrightarrow \quad& \lambda_m-\lambda_{m+1} \leq \sigma_m-\sigma_{m+1} + \sqrt{2} \| \E Z -\bar{Z}_n \|_F
\end{align*}
where $\sigma_m$ and $\sigma_{m+1}$ denote the $m$-th and ${m+1}$-th eigenvalues of $\E Z.$ If $\pi(\E Z)$ is not unique, then $\sigma_m=\sigma_{m+1}.$ By \thref{proconf}, 
\begin{equation*}
\| \E Z -\bar{Z}_n \|_F < \sqrt{\tfrac{1}{\alpha n} \bigl(1-\|\bar{Z}_n\|_F^2+\tfrac{1}{\alpha n} \bigr)}
\end{equation*}
with probability at least $(1-\alpha).$

For \emph{(iii)},
\begin{align*}
\P \left( \lambda_m-\lambda_{m+1} \geq \sqrt{\tfrac{2 \left( 1-\|\bar{Z}_n\|_F^2+\tfrac{1}{\alpha n} \right) }{\alpha n}} \right) & \geq \P \left( \sigma_m-\sigma_{m+1} - \sqrt{2} \| \E Z -\bar{Z}_n \|_F \geq \sqrt{\tfrac{2 \left(1+\tfrac{1}{\alpha n} \right)}{\alpha n}} \right) \\
& = \P \left( \|\E Z -\bar{Z}_n \|_F \leq \tfrac{\sigma_m-\sigma_{m+1}}{\sqrt{2}} - \sqrt{\tfrac{1+\tfrac{1}{\alpha n} }{\alpha n}} \right) \\
& \geq 1 - e^{-\tfrac{1}{2} \left( \sqrt{n}\tfrac{\sigma_m-\sigma_{m+1}}{\sqrt{2}} - \sqrt{\tfrac{1+\tfrac{1}{\alpha n}}{\alpha}} - \sqrt{1-\|\E Z\|^2} \right)^2 } \\
& \longrightarrow 1,
\end{align*} using results of \citep{BLM}.
\end{proof}

\section{Application to shape spaces}\label{Applications}

Consider the projective shape space $P\Sigma_m^k,$ $k>m+2$ consisting of all shapes with a projective frame in the first $m+2$ points. A result of \citet{MP} is a diffeomorphism $P\Sigma_m^k \cong \left(\rp^m\right)^q$ with $q=k-m-2,$ hence non-asymptotic confidence sets can be computed using \thref{grconfreg}. For the sake of readability, let $q=1$ throughout this example. Recall that the Euclidean space $\R^m$ can be embedded in $\rp^m$ preserving collinearity, e.g.\ by using homogeneous coordinates 
$$x=(x_1,\ldots,x_m)^t\mapsto[x:1]:=[x_1:\ldots:x_m:1]$$ 
missing only the hyperplane at infinity. Using this embedding, a point $[X_1:\ldots:X_m:X_{m+1}]\in\rp^m$ has a representative in $\mathbb{R}^m$ if and only if $X_{m+1}\neq0.$ Given data in $P\Sigma_m^k,$ we want to illustrate a confidence set for the extrinsic mean in $\mathbb{R}^m$ if the extrinsic sample mean $\mu$ is not in the infinity hyperplane, i.e.\ has a representative in $\mathbb{R}^m.$

Define $\tilde{x}:=(x^t,1)^t\in\mathbb{R}^{m+1}$ for $x\in\mathbb{R}^m$ and let $[z:1]=\mu.$ Then for all $x\in\mathbb{R}^m$ with 
\begin{equation*}
d^2([x:1],[z:1])=2-\frac{2}{\left(\|x\|^2_2+1\right)\left(\|z\|^2_2+1\right)}\left(x^tz+1\right)^2<\delta_n^2
\end{equation*} the following holds: 
\begin{equation*}\left(2-\delta_n^2\right)\left(\|z\|^2_2+1\right)<2\left(x^tz+1\right)^2-\left(2-\delta_n^2\right)\|x\|^2_2\left(\|z\|^2_2+1\right).
\end{equation*} Hence, visualizing confidence sets for extrinsic means in $\Sigma_m^k$ is understanding quadrics in $\mathbb{R}^m.$ Particularly in the case $m=2,$ these quadrics represent cone sections such that the confidence sets are the ``interior'' of these.

\begin{figure}[t]
\centering
\includegraphics[scale=0.6]{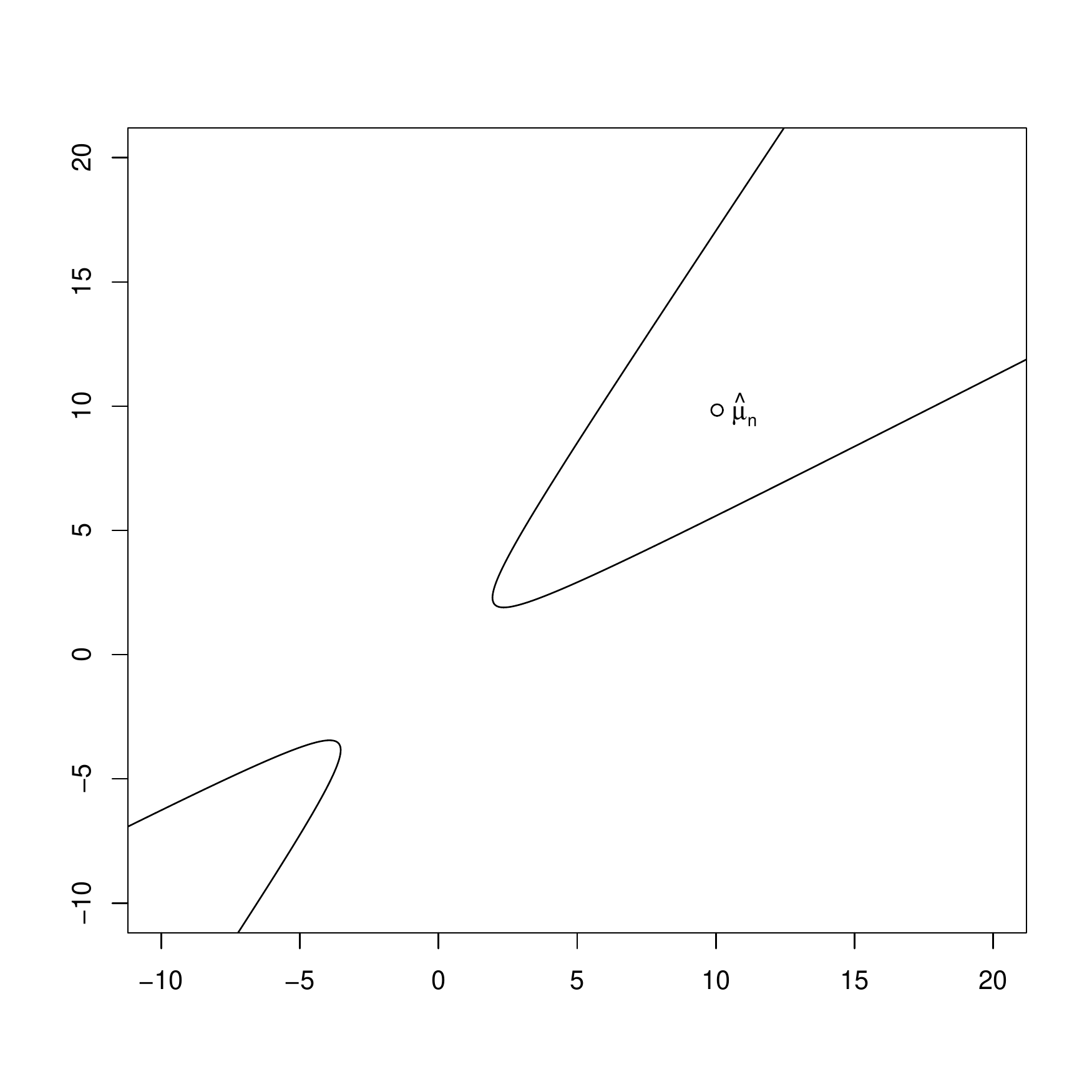} 

\caption{Here, the confidence set to the synthesized data from \hyperref[Applications]{Section 5} is plotted. This confidence set is a subset of $\R^2$ when inhomogeneous coordinates are used. The confidence set contains the extrinsic sample mean $\hat{\mu}_n,$ and the boundary of the confidence region is a hyperbola.}
\label{RP2}

\end{figure}

A synthesized data set will be used to show the value of this discussion. Consider $n=100$ points $(10+x_1, 10+x_2)\in\R^2$ with $x_1,x_2$ being random real numbers generated by \texttt{R} with univariate distribution (mean 0, standard derivation 1). Embedding of $\mathbb{R}^2$ in $\rp^2$ produces data in $\rp^2$ resp.\ in $\herm(3).$ The resulting Euclidean sample mean is 
\begin{equation*}
\begin{pmatrix} 0.5075 & 0.4917 & 0.0503 \\ 0.4917 & 0.4875 & 0.0492 \\ 0.0503 & 0.0492 & 0.005 \end{pmatrix},
\end{equation*}
 hence the extrinsic sample mean is 
\begin{equation*}
 \begin{bmatrix} 10.0447 \\ 9.8422 \\ 1 \end{bmatrix}
\end{equation*}
 and $\delta_{100}=0.3713.$ See \hyperref[RP2]{Figure\thref{RP2}} for the visualization.


\section{Discussion and outlook}

We showed how to construct non-asymptotic and rate-optimal confidence sets for the extrinsic population mean for i.i.d.\ data on spheres resp.\ Grassmann manifold. Unfortunately, these are too big in comparison to asymptotic confidence sets to be of practical use due to the fact that the Chebychev inequality has rather loose bounds and is not sharp for bounded random variables. Therefore, sharper inequalities for multivariate, bounded random variables would result in smaller confidence regions. Additionally, one would like to take the sample covariance into account as in \citep{HKW}. Unfortunately, the construction of the confidence sets gets tougher for the known multivariate mass concentration inequalities whence the construction of these is one of the aims of future research.

\bibliographystyle{dcu}
\bibliography{confreg}

\end{document}